\documentclass[12pt]{amsart}

\usepackage[letterpaper,margin=1.1in]{geometry}
\usepackage{amssymb} 
\usepackage{hyperref}

\newtheorem{maintheorem}{Theorem}

\newtheorem{theorem}{Theorem}[section]
\newtheorem{lemma}[theorem]{Lemma}
\newtheorem{corollary}[theorem]{Corollary}

\theoremstyle{definition}
\newtheorem{definition}[theorem]{Definition}

\theoremstyle{remark}
\newtheorem{remark}[theorem]{Remark}
\newtheorem{problem}[theorem]{Problem}

\newcommand{\RR}{\mathbb{R}}
\newcommand{\Haus}{\mathcal{H}}
\newcommand{\dist}{\mathop\mathrm{dist}\nolimits}
\newcommand{\diam}{\mathop\mathrm{diam}\nolimits}
\newcommand{\lD}[1]{\mathop{\underline{D}^{#1}}\nolimits}
\newcommand{\uD}[1]{\mathop{\overline{D}^{\,#1}}\nolimits}
\newcommand{\res}{\hbox{ {\vrule height .22cm}{\leaders\hrule\hskip.2cm} }} 

\numberwithin{equation}{section}

\begin{document}

\title{Two sufficient conditions for rectifiable measures}
\thanks{M.~ Badger was partially supported by an NSF postdoctoral fellowship DMS 1203497.
R.~ Schul was partially supported by NSF  DMS 1361473.}
\date{June 29, 2015}
\author{Matthew Badger \and Raanan Schul}
\address{Department of Mathematics\\ University of Connecticut\\ Storrs, CT 06269-3009}
\email{matthew.badger@uconn.edu}
\address{Department of Mathematics\\ Stony Brook University\\ Stony Brook, NY 11794-3651}
\email{schul@math.sunysb.edu}

\begin{abstract} We identify two sufficient conditions for locally finite Borel measures on $\RR^n$ to give full mass to a countable family of Lipschitz images of $\RR^m$. The first condition, extending a prior result of Pajot, is a sufficient test in terms of $L^p$ affine approximability for a locally finite Borel measure $\mu$ on $\RR^n$ satisfying the global regularity hypothesis $$\limsup_{r\downarrow 0} \mu(B(x,r))/r^m <\infty\quad \text{at $\mu$-a.e.~$x\in\RR^n$}$$ to be $m$-rectifiable in the sense above. The second condition is an assumption on the growth rate of the 1-density that ensures a locally finite Borel measure $\mu$ on $\RR^n$ with $$\lim_{r\downarrow 0} \mu(B(x,r))/r=\infty\quad\text{at $\mu$-a.e.~$x\in\RR^n$}$$ is 1-rectifiable. \end{abstract}

\maketitle

\section{Introduction}

In the treatise \cite{Federer} on geometric measure theory, Federer supplies the following general notion of rectifiability with respect to a measure. Let $1\leq m\leq n-1$ be integers. Let $\mu$ be a \emph{Borel measure} on $\RR^n$, i.e.~a Borel regular outer measure on $\RR^n$. Then \emph{$\RR^n$ is countably $(\mu,m)$ rectifiable} if there exist countably many Lipschitz maps $f_i:[0,1]^m\rightarrow\RR^n$ such that $\mu$ assigns full measure to the images sets $f_i([0,1]^m)$, i.e. $$\mu\left(\RR^n\setminus\bigcup_{i=1}^\infty f_i([0,1]^m)\right)=0.$$ When $m=1$, each set $\Gamma_i=f_i([0,1])$ is a \emph{rectifiable curve}. Below we shorten Federer's terminology, saying that \emph{$\mu$ is $m$-rectifiable} if $\RR^n$ is countably $(\mu,m)$ rectifiable.

Two well studied subclasses of rectifiable measures are
Hausdorff measures on rectifiable sets and absolutely continuous rectifiable measures.
Given any Borel measure $\mu$ on $\RR^n$ and Borel set $E\subseteq\RR^n$, define the measure $\mu\res E$ (``$\mu$ restricted to $E$") by the rule $\mu\res E(F)=\mu(E\cap F)$ for all Borel sets $F\subseteq\RR^n$. We call a Borel set $E\subseteq\RR^n$ an \emph{$m$-rectifiable set} if $\Haus^m\res E$ is an $m$-rectifiable measure, where $\Haus^m$ denotes the $m$-dimensional Hausdorff measure on $\RR^n$. One may think of an $m$-rectifiable set $E$ as an $m$-rectifiable measure by identifying $E$ with the measure $\Haus^m\res E$. More generally, we say that an $m$-rectifiable measure $\mu$ on $\RR^n$ is \emph{absolutely continuous} if $\mu\ll\Haus^m$, i.e.~ $\mu(E)=0$ whenever $E\subset\RR^n$ and $\Haus^m(E)=0$.

It is a remarkable fact that rectifiable sets and absolutely continuous rectifiable measures can be identified by the asymptotic behavior of the measures on small balls.

\begin{definition}[Hausdorff density] Let $B(x,r)$ denote the closed ball in $\RR^n$ with center $x\in\RR^n$ and radius $r>0$. For each positive integer $m\geq 1$, let $\omega_m=\Haus^m(B^m(0,1))$ denote the volume of the unit ball in $\RR^m$. For all locally finite Borel measures $\mu$ on $\RR^n$, we define the \emph{lower Hausdorff $m$-density} $\lD{m}(\mu,\cdot)$ and  \emph{upper Hausdorff $m$-density} $\uD{m}(\mu,\cdot)$ by \begin{equation*} \lD{m}(\mu,x):=\liminf_{r\rightarrow 0} \frac{\mu(B(x,r))}{\omega_mr^m}\in[0,\infty]\end{equation*} and \begin{equation*} \uD{m}(\mu,x):=\limsup_{r\rightarrow 0} \frac{\mu(B(x,r))}{\omega_mr^m}\in[0,\infty]\end{equation*} for all $x\in\RR^n$.
If $\lD{m}(\mu,x)=\uD{m}(\mu,x)$ for some $x\in\RR^n$, then we write $D^m(\mu,x)$ for the common value and call $D^m(\mu,x)$ the \emph{Hausdorff $m$-density of $\mu$ at $x$}.\end{definition}

\begin{theorem}[\cite{Mattila75}]\label{t:M} Let $1\leq m\leq n-1$. Suppose $E\subset\RR^n$ is Borel and $\mu=\Haus^m\res E$ is locally finite. Then $\mu$ is $m$-rectifiable if and only if the Hausdorff $m$-density of $\mu$ exists and $D^m(\mu,x)=1$ at $\mu$-a.e.~ $x\in\RR^n$.\end{theorem}

\begin{theorem}[\cite{Preiss}]\label{t:P} Let $1\leq m\leq n-1$. If $\mu$ is a locally finite Borel measure on $\RR^n$, then $\mu$ is $m$-rectifiable and $\mu\ll \Haus^m$  if and only if the Hausdorff $m$-density of $\mu$ exists and $0<D^m(\mu,x)<\infty$ at $\mu$-a.e.~ $x\in\RR^n$.\end{theorem}

\begin{remark} For any locally finite Borel measure $\mu$ on $\RR^n$: \begin{align}\notag \mu\ll\Haus^m\quad&\Longleftrightarrow\quad \uD{m}(\mu,x)<\infty\text{ at $\mu$-a.e.~$x\in\RR^n$; and,}\\ \label{e:rect-ld} \mu\text{ is $m$-rectifiable}\quad&\,\,\Longrightarrow \quad \lD{m}(\mu,x)>0 \text{ at $\mu$-a.e.~$x\in\RR^n$}.\end{align}See \cite[Chapter 6]{Mattila} and \cite[Lemma 2.7]{BS}.
\end{remark}

There are several other characterizations of rectifiable sets and absolutely continuous rectifiable measures (e.g. in terms of projections or tangent measures); see Mattila \cite{Mattila} for a full survey of results through 1993. Further investigations on rectifiable sets and absolutely continuous rectifiable measures include \cite{Pajot96,Pajot97,Leger,Lerman,Tolsa12,CGLT,TT, Tolsa-square,ADT1,Bate-Li,Buet,ADT2,AT,Tolsa-n}.

The first result of this note is an extension of Pajot's theorem on rectifiable sets \cite{Pajot97} to absolutely continuous rectifiable measures. To state these results, we must recall the notion of an $L^p$ beta number from the theory of quantitative rectifiability.

\begin{definition}[$L^p$ beta numbers] \label{def:beta}
Let $1\leq m\leq n-1$ and let $1\leq p<\infty$. For every locally finite Borel measure $\mu$ on $\RR^n$ and bounded Borel set $Q\subset\RR^n$, define  $\beta^{(m)}_p(\mu,Q)$ by
\begin{equation} \label{e:beta-p}
   \beta^{(m)}_p(\mu,Q)^p:=\inf_\ell \int_Q \left(\frac{\dist(x,\ell)}{\diam Q}\right)^p \frac{d\mu(x)}{\mu(Q)}\in[0,1],
\end{equation}
where $\ell$ in the infimum ranges over all $m$-dimensional affine planes in $\RR^n$. If $\mu(Q)=0$, then we interpret (\ref{e:beta-p}) as $\beta^{(m)}_p(\mu,Q)=0$.\end{definition}

\begin{remark}Beta numbers (of sets) were introduced by Jones \cite{Jones-TST} to characterize subsets of rectifiable curves in the plane and are now often called \emph{Jones beta numbers}. The $L^p$ variant in Definition \ref{def:beta} originated in the fundamental work of David and Semmes on uniformly rectifiable sets \cite{DS91,DS93}
with the normalization appearing in \eqref{e:reg-normalization}. The normalization of $\beta^{(m)}_p(\mu,Q)$ presented in Definition \ref{def:beta} is not new; see e.g.~\cite{Lerman}.

When $Q=B(x,r)$, some sources (e.g.~\cite{DS91,DS93,Pajot97}) define $L^p$ beta numbers using the alternate normalization \begin{equation}\label{e:reg-normalization}
\widetilde\beta^{(m)}_p(\mu,B(x,r))^p:=\inf_{\ell} \int_{B(x,r)} \left(\frac{\dist(x,\ell)}{r}\right)^p \frac{d\mu(x)}{r^m}\in[0,\infty),
\end{equation}
where $\ell$ in the infimum again ranges over all $m$-dimensional affine planes in $\RR^n$. However,  $\beta^{(m)}_p(\mu,B(x,r))$ and $\widetilde{\beta}^{(m)}_p(\mu,B(x,r))$ are quantitatively equivalent at locations and scales where $\mu(B(x,r))\sim r^m$. We have freely translated beta numbers in theorem statements quoted from other sources to the convention of Definition \ref{def:beta}, which is better suited for generic locally finite Borel measures. \end{remark}

\begin{theorem}[\cite{Pajot97}] \label{t:pajot} Let $1\leq m\leq n-1$ and let \begin{equation}\label{e:p-range} \left\{\begin{array}{ll}1\leq p<\infty &\text{ if $m=1$ or $m=2$,}\\ 1 \leq p<2m/(m-2) &\text{ if $m\geq 3$.}\end{array}\right.\end{equation} Assume that $K\subset\RR^n$ is compact and $\mu=\Haus^m\res K$ is a finite measure. If $\lD{m}(\mu,x)>0$ at $\mu$-a.e.~$x\in\RR^n$ and \begin{equation}\label{e:pajot-bp}\int_0^1 \beta^{(m)}_p(\mu, B(x,r))^2\,\frac{dr}{r}<\infty\quad\text{at $\mu$-a.e.~$x\in \RR^n$},\end{equation} then $\mu$ is $m$-rectifiable.
\end{theorem}

In \S2, we note the following extension of Pajot's theorem. Also, see Theorem 2.1.

\begin{maintheorem}\label{t:suff} Let $1\leq m\leq n-1$ and let $1\leq p<\infty$ satisfy \eqref{e:p-range}. Assume that $\mu$ is a locally finite Borel measure on $\RR^n$ such that $\mu\ll\Haus^m$. If $\lD{m}(\mu,x)>0$ at $\mu$-a.e.~ $x\in\RR^n$ and \eqref{e:pajot-bp} holds, then $\mu$ is $m$-rectifiable.\end{maintheorem}

In a forthcoming paper, Tolsa \cite{Tolsa-n} proves that \eqref{e:pajot-bp} is a necessary condition for an absolutely continuous measure to be rectifiable. Together with Theorem A and  \eqref{e:rect-ld}, this result provides a full characterization of absolutely continuous rectifiable measures in terms of the beta numbers and lower Hausdorff density of a measure.

\begin{theorem}[\cite{Tolsa-n}] \label{t:tolsa} Let $1\leq m\leq n-1$ and let $1\leq p\leq 2$. If $\mu$ is $m$-rectifiable and $\mu\ll\Haus^m$, then \eqref{e:pajot-bp} holds.\end{theorem}

\begin{corollary}\label{c:B} Let $1\leq m\leq n-1$ and let $1\leq p\leq 2$. If $\mu$ is a locally finite Borel measure on $\RR^n$ such that $\mu\ll\Haus^m$, then the following are equivalent: \begin{itemize} \item $\mu$ is $m$-rectifiable;
\item $\lD{m}(\mu,x)>0$ at $\mu$-a.e.~ $x\in\RR^n$ and \eqref{e:pajot-bp} holds.
\end{itemize}\end{corollary}

In a companion paper to \cite{Tolsa-n}, Azzam and Tolsa \cite{AT} prove that in the case $p=2$, Theorem \ref{t:suff} holds with the hypothesis $\lD{m}(\mu,x)>0$ at $\mu$-a.e.~$x\in\RR^n$ on the lower density replaced by a weaker assumption $\uD{m}(\mu,x)>0$ at $\mu$-a.e.~$x\in\RR^n$ on the upper density.

For general $m$-rectifiable measures that are allowed to be singular with respect to $\Haus^m$, the following basic problem in geometric measure theory is still open.

\begin{problem} \label{p:gmt} For all $1\leq m\leq n-1$, find necessary and sufficient conditions in order for a locally finite Borel measure $\mu$ on $\RR^n$ to be $m$-rectifiable. (Do not assume that $\mu\ll\Haus^m$.) \end{problem}

Partial progress on Problem \ref{p:gmt} has recently been made in  \cite{GKS,BS, AM15} in the case $m=1$. In \cite{GKS}, Garnett, Killip, and Schul exhibit a family $(\nu_\delta)_{0<\delta\leq \delta_0}$ of self-similar locally finite Borel measures on $\RR^n$, which are \begin{itemize}
 \item \emph{doubling}: $0<\nu_\delta(B(x,r))\leq C_\delta\,\nu_\delta(B(x,r/2))<\infty$ for all $x\in\RR^n$ and $r>0$;
 \item \emph{badly linearly approximable}: $\beta^{(1)}_2(\nu_\delta,B(x,r))\geq c_\delta>0$ for all $x\in\RR^n$ and $r>0$;
 \item \emph{singular}: $D^1(\nu_\delta,x)=\infty$ at $\nu_\delta$-a.e.~$x\in\RR^n$ (hence $\nu_\delta\perp \Haus^1$); and,
  \item \emph{1-rectifiable}: $\nu_\delta(\RR^n\setminus\bigcup_{i}\Gamma_i)=0$ for some countable family of rectifiable curves $\Gamma_i$.
\end{itemize} In \cite{BS}, Badger and Schul identify a pointwise necessary condition for an arbitrary locally finite Borel measure $\mu$ on $\RR^n$ to be 1-rectifiable.

\begin{theorem}[{\cite[Theorem A]{BS}}] \label{t:BS-n}Let $n\geq 2$ and let $\Delta$ be a system of closed or half-open dyadic cubes in $\RR^n$ of side length at most 1. If $\mu$ is a locally finite Borel measure on $\RR^n$ and $\mu$ is 1-rectifiable, then \begin{equation*}\label{e:Jtilde} \sum_{Q\in\Delta} \beta^{(1)}_2(\mu,3Q)^2\, \frac{\diam Q}{\mu(Q)}\chi_Q(x)<\infty\quad\text{at $\mu$-a.e.~$x\in\RR^n$}.\end{equation*}
\end{theorem}

The second result of this note is a sufficient condition for a measure $\mu$ with $D^1(\mu,x)=\infty$ at $\mu$-a.e.~$x\in\RR^n$ to be 1-rectifiable.

\begin{maintheorem}\label{t:C} Let $n\geq 2$ and let $\Delta$ be a system of half-open dyadic cubes in $\RR^n$ of side length at most 1. If $\mu$ is a locally finite Borel measure on $\RR^n$ and \begin{equation}\label{e:C}\sum_{Q\in\Delta}\frac{\diam Q}{\mu(Q)}\chi_Q(x)<\infty\quad\text{at $\mu$-a.e. $x\in\RR^n$},\end{equation} then $\mu$ is 1-rectifiable, and moreover, there exist a countable family of rectifiable curves $\Gamma_i$ and Borel sets $B_i\subseteq \Gamma_i$ such that $\Haus^1(B_i)=0$ for all $i\geq 1$ and $\mu\left(\RR^n\setminus\bigcup_{i=1}^\infty B_i\right)=0$. \end{maintheorem}

Together Theorem \ref{t:BS-n} and Theorem \ref{t:C} provide a full characterization of 1-rectifiability of measures with ``pointwise large beta number" \eqref{e:large-beta}. Examples of measures that satisfy this beta number condition include the measures $(\nu_\delta)_{0<\delta\leq \delta_0}$ from \cite{GKS}, or more generally, any doubling measure $\mu$ on $\RR^n$ whose support is $\RR^n$.

\begin{corollary}\label{c:D} Let $n\geq 2$ and let $\Delta$ be a system of half-open dyadic cubes in $\RR^n$ of side length at most 1. If $\mu$ is a locally finite Borel measure such that \begin{equation}\label{e:large-beta} \liminf_{\stackrel{Q\in\Delta, x\in Q}{\diam Q\rightarrow 0}} \beta^{(1)}_2(\mu,3Q)>0\quad\text{at $\mu$-a.e.~ $x\in\RR^n$},\end{equation} then $\mu$ is 1-rectifiable if and only if \eqref{e:C} holds.\end{corollary}

Finally, we note that in recent work Azzam and Mourgoglou \cite{AM15} give a weaker condition for 1-rectifiability of a doubling measure with connected support.
\begin{theorem}[\cite{AM15}] \label{AM15} Let $\mu$ be a doubling measure whose support is a topologically connected metric space $X$ and let $E\subseteq X$ be compact. Then $\mu\res E$ is 1-rectifiable if and only if $\lD{1}(\mu,x)>0$ for $\mu$-a.e.~$x\in E$.
\end{theorem}

When applied to a doubling measure $\mu$ on $\RR^n$ whose support is $\RR^n$, Corollary \ref{c:D} and Theorem \ref{AM15} imply that if $\lD{1}(\mu,x)>0$ at $\mu$-a.e. $x\in\RR^n$, then \eqref{e:C} holds.

\bigskip

The remainder of this note is split into two sections. We prove Theorem \ref{t:suff} in \S2 and we prove Theorem \ref{t:C} in \S3.

\section{Proof of Theorem \ref{t:suff}}

We show how to reduce Theorem \ref{t:suff} to Theorem \ref{t:pajot} using standard geometric measure theory techniques; see Chapters 1, 2, 4, and 6 of \cite{Mattila} for general background. In fact, we will establish the following ``localized version" of Theorem \ref{t:suff}.

\begin{theorem} Let $1\leq m\leq n-1$ and let \begin{equation}\label{e:p-range2} \left\{\begin{array}{ll}1\leq p<\infty &\text{ if $m=1$ or $m=2$,}\\ 1 \leq p<2m/(m-2) &\text{ if $m\geq 3$.}\end{array}\right.\end{equation} If $\mu$ is a locally finite Borel measure on $\RR^n$ such that \begin{equation*}J_p(\mu,x):=\int_0^1 \beta^{(m)}_p(\mu,B(x,r))^2\,\frac{dr}{r}<\infty\quad\text{at $\mu$-a.e.~$x\in\RR^n$},\end{equation*} then $\mu\res\left\{x\in\RR^n: 0<\lD{m}(\mu,x)\leq \uD{m}(\mu,x)<\infty\right\}$ is $m$-rectifiable.
\end{theorem}

\begin{proof} Without loss of generality, we assume for the duration of the proof that $\Haus^m$ is normalized so that $\omega_m=\Haus^m(B^m(0,1))=2^m$. This is the convention used in \cite{Mattila}.

Suppose that $1\leq m\leq n-1$, let $p$ belong to the range (\ref{e:p-range2}), and let $\mu$ be a locally finite Borel measure on $\RR^n$ such that $J_p(\mu,x)<\infty$ at $\mu$-a.e.~ $x\in\RR^n$. Define $$A:=\left\{x\in \RR^n: 0<\lD{m}(\mu,x)\leq \uD{m}(\mu,x)<\infty\right\}.$$ Also, for each pair of integers $j,k\geq 1$, define $$A(j,k):=\left\{x\in B(0,2^k): 2^{-j}r^m \leq \mu(B(x,r))\leq 2^jr^m\text{ for all }0<r\leq 2^{-k}\right\}.$$ Then $\overline{A(j,k)}$ is compact and  $\overline{A(j,k)}\subseteq A(j+1,k+1)$ for all $j,k\geq 1$. Also note that $$A=\bigcup_{j,k=1}^\infty A(j,k) = \bigcup_{j,k=1}^\infty \overline{A(j,k)},$$ Thus, to prove that $\mu\res A$ is $m$-rectifiable, it suffices to verify that $\mu\res\overline{A(j,k)}$ is $m$-rectifiable for all $j,k\geq 1$.

Fix any $j,k\geq 1$ and set $K:=\overline{A(j,k)}$, $\nu:=\mu\res K$, and $\sigma:=\Haus^m\res K$. In order to prove that $\nu$ is $m$-rectifiable, it is enough to show that $\nu\ll\sigma\ll\nu$ and $\sigma$ is $m$-rectifiable. By Theorem 6.9 in \cite{Mattila},  since $2^{-j-1-m}\leq\uD{m}(\mu,x)\leq 2^{j+1-m}$ for all $x\in K$, we have \begin{equation}\label{e:nu} \nu(B(x,r))=\mu(K\cap B(x,r))\leq 2^{j+1}\Haus^m(K\cap B(x,r))=2^{j+1}\sigma(B(x,r))\end{equation} and \begin{equation}\label{e:sigma} \sigma(B(x,r))=\Haus^m(K\cap B(x,r)) \leq 2^{j+1+m} \mu(K\cap B(x,r))=2^{j+1+m}\nu(B(x,r))\end{equation} for all $x\in\RR^n$ and $r>0$. Note that $$\sigma(\RR^n)=\sigma(B(0,2^k)) \leq 2^{j+1+m}\mu(B(0,2^k))<\infty,$$ since $\mu$ is locally finite. That is, $\sigma$ is a finite measure. Thus, $\nu$ and $\sigma$ are mutually absolutely continuous by \eqref{e:nu}, \eqref{e:sigma}, and Lemma 2.13 in \cite{Mattila}. Now,
 \begin{equation}\label{e:a1} \sigma(B(x,r))\leq 2^{j+1+m}\mu(B(x,r))\leq 2^{2j+2+m}r^m\quad\text{for all }x\in K\text{ and }0<r\leq 2^{-k-1}.\end{equation} On the other hand,
let $K'$ denote the set of $x\in K$ such that $$2\nu(B(x,r))=2\mu(K\cap B(x,r))\geq \mu(B(x,r))\quad\text{for all }0<r\leq r_x$$ for some $r_x\leq 2^{-k-1}$. Then $\sigma(\RR^n\setminus K')=0$, because $\nu(\RR^n\setminus K')=\mu(K\setminus K')=0$, and \begin{equation}\label{e:a2}\sigma(B(x,r))\geq 2^{-j-2}\mu(B(x,r))\geq 2^{-2j-3}r^m\quad\text{for all }x\in K'\text{ and }0<r\leq r_x.\end{equation} In particular, $\lD{m}(\sigma,x)\geq c(m,j)>0$ at $\sigma$-a.e.~$x\in\RR^n$. To conclude that $\sigma$ is $m$-rectifiable using Theorem \ref{t:pajot}, it remains to verify $J_p(\sigma,x)<\infty$ at $\sigma$-a.e.~$x\in\RR^n$.

By \eqref{e:a1} and \eqref{e:a2}, there exists a constant $C=C(m,j)<\infty$ such that $$C^{-1}\leq \frac{\nu(B(x,r))}{\sigma(B(x,r))}\leq C\quad\text{for all $0<r\leq r_x$ at $\sigma$-a.e.~ $x\in\RR^n$}.$$ Thus, by differentiation of Radon measures, we can write $d\nu=f\, d\sigma$, where $f\in L^1_{\rm{loc}}(d\sigma)$ and $C^{-1} \leq f(x) \leq C$ at $\sigma$-a.e. $x\in\RR^n$. Therefore, at $\sigma$-a.e.~ $x\in \RR^n$, for every $0<r\leq r_x$ and for every $m$-dimensional affine plane $\ell$, \begin{align*}\int_{B(x,r)} \left(\frac{\dist(y,\ell)}{\diam B(x,r)}\right)^p \frac{d\sigma(y)}{\sigma(B(x,r))}&\leq C^2 \int_{B(x,r)} \left(\frac{\dist(y,\ell)}{\diam B(x,r)}\right)^p \frac{d\nu(y)}{\nu(B(x,r))}\\
&\leq 2C^2 \int_{B(x,r)} \left(\frac{\dist(y,\ell)}{\diam B(x,r)}\right)^p \frac{d\mu(y)}{\mu(B(x,r))}. \end{align*} Thus, $\beta^{(m)}_p(\sigma,B(x,r))^2 \leq \left(2C^2\right)^{2/p} \beta^{(m)}_p(\mu,B(x,r))^2$ for all $0<r\leq r_x$ at $\sigma$-a.e.~$x\in\RR^n$. Since $J_p(\mu,x)<\infty$ at $\mu$-a.e.~$x\in \RR^n$ and $\sigma\ll\mu$, it follows that $J_p(\sigma,x)<\infty$ at $\sigma$-a.e.~$x\in\RR^n$. Finally, since $K$ is compact, $\sigma=\Haus^m\res K$ is finite, and $\lD{m}(\sigma,x)>0$ and $J_p(\sigma,x)<\infty$ at $\sigma$-a.e.~$x\in\RR^n$, we conclude that $\sigma$ is $m$-rectifiable by Theorem \ref{t:pajot}. As noted above, this implies that $\nu=\mu\res\overline{A(j,k)}$ is $m$-rectifiable for all $j,k\geq 1$, and therefore, $\mu\res A$ is $m$-rectifiable.
\end{proof}

\section{Proof of Theorem \ref{t:C}}

For every Borel measure $\mu$ on $\RR^n$, define the quantity $$S(\mu,x):= \sum_{Q\in\Delta} \frac{\diam Q}{\mu(Q)}\chi_Q(x)\in[0,\infty]\quad\text{for all }x\in\RR^n,$$ where $\Delta$ denotes any system of \emph{half-open} dyadic cubes in $\RR^n$ of side length at most $1$. Theorem \ref{t:C} is a special case of the following statement.

\begin{theorem}\label{t:fast} Let $n\geq 2$. If $\mu$ is a locally finite Borel measure on $\RR^n$, then $$\rho:=\mu \res\left\{x\in\RR^n: S(\mu,x)<\infty\right\}$$ is $1$-rectifiable. Moreover, there exists a countable family of rectifiable curves $\Gamma_i\subset\RR^n$ and Borel sets $B_i\subseteq \Gamma_i$ such that $\Haus^1(B_i)=0$ for all $i\geq 1$ and $\rho(\RR^n\setminus\bigcup_{i=1}^\infty B_i)=0$.\end{theorem}

We start with a lemma, which will be used to organize the proof of Theorem \ref{t:fast}.

\begin{lemma}\label{l:good-bad} Let $n\geq 1$ and let $\mu$ be a locally finite Borel measure on $\RR^n$. Given $Q_0\in\Delta$ such that $\eta:=\mu(Q_0)>0$ and $N<\infty$, let $$A:=\{x\in Q_0: S(\mu,x)\leq N\}.$$ For all $0<\varepsilon<1/\eta$, the set of dyadic cubes $Q\subseteq Q_0$ can be partitioned into \emph{good cubes} and \emph{bad cubes} with the following properties: \begin{enumerate}
\item every child of a bad cube is a bad cube;
\item the set $B:=A\setminus\bigcup\{Q:Q\subseteq Q_0$ is a bad cube$\}$ satisfies $\mu(B)\geq (1-\varepsilon\eta)\mu(A)$;
\item $\sum\diam Q< N/\varepsilon$, where the sum ranges over all good cubes $Q\subseteq Q_0$.
\end{enumerate}\end{lemma}

\begin{proof} Suppose that $n$, $\mu$, $Q_0$, $\eta$, $N$, and $A$ are given as above and let $\varepsilon>0$. If $\mu(A)=0$, then we may declare every dyadic cube $Q\subseteq Q_0$ to be a bad cube and the conclusion of the lemma hold trivially. Thus, suppose that $\mu(A)>0$. Declare that a dyadic cube $Q\subseteq Q_0$ is a \emph{bad cube} if there exists a dyadic cube $R\subseteq Q_0$ such that $Q\subseteq R$ and $\mu(A\cap R)\leq \varepsilon \mu(A)\mu(R)$. We call a dyadic cube $Q\subseteq Q_0$ a \emph{good cube} if $Q$ is not a bad cube. Property (1) is immediate. To check property (2),
observe that
$$\mu(A\setminus B) \leq \sum_{\text{maximal bad }Q\subseteq Q_0} \mu(A\cap Q) \leq \varepsilon\mu(A)\sum_{\text{maximal bad }Q\subseteq Q_0} \mu(Q) \leq \varepsilon\mu(A)\mu(Q_0),$$ where the last inequality follows because the maximal bad cubes are pairwise disjoint (since $\Delta$ is composed of half-open cubes). Recalling $\mu(Q_0)=\eta$, it follows that $$\mu(B)=\mu(A)-\mu(A\setminus B)\geq  (1-\varepsilon\eta)\mu(A).$$ Thus, property (2) holds. Finally, since $S(\mu,x)\leq N$ for all $x\in A$, $$N\mu(A)
 \geq \int_A S(\mu,x)\,d\mu(x) \geq \sum_{Q\subseteq Q_0} \diam Q\, \frac{\mu(A\cap Q)}{\mu(Q)}> \varepsilon\mu(A)\sum_{\text{good }Q\subseteq Q_0} \diam Q,$$ where we interpret $\mu(A\cap Q)/\mu(Q)=0$ if $\mu(Q)=0$. Because $\mu(A)>0$, it follows that $$\sum_{\text{good }Q\subseteq Q_0} \diam Q <\frac{N}{\varepsilon}.$$ This verifies property (3).
\end{proof}

\begin{lemma}\label{l:tree} Let $n\geq 2$ and let $\mu$ be a locally finite Borel measure on $\RR^n$. If $$\mu(\{x\in Q_0:S(\mu,x)\leq N\})>0\quad\text{for some $Q_0\in\Delta$ and $N<\infty$},$$ then for all $0<\varepsilon<1/\mu(Q_0)$ the set $B=B(\mu,Q_0,N,\varepsilon)$ described in Lemma \ref{l:good-bad} lies in a rectifiable curve $\Gamma$ with $\Haus^1(\Gamma)< N/2\varepsilon$ and $\Haus^1(B)=0$.\end{lemma}

\begin{proof}
Let $n\geq 2$ and let $\mu$ be a locally finite Borel measure on $\RR^n$. Suppose $\mu(A)>0$ for some $Q_0\in\Delta$ and $N<\infty$, where $A=\{x\in Q_0: S(\mu,x)\leq N\}$. Then $\eta:=\mu(Q_0)>0$, as well. Given any $0<\varepsilon<1/\eta$, let $B=B(\mu,Q_0,N,\varepsilon)$ denote the set from Lemma \ref{l:good-bad}. Since
$\varepsilon$ is small enough such that $\mu(B)\geq (1-\varepsilon\eta)\mu(A)>0$, the cube $Q_0$ is a good cube.
Construct a connected set $T\subset \RR^n$ by drawing a (closed) straight line segment
$\ell_Q$ from the center of each good cube $Q\subsetneq Q_0$ to the center of its parent, which is also a good cube. Let $\overline{T}$ denote the closure of $T$. For all $\delta>0$, $$\overline{T} \subseteq \bigcup_{\stackrel{\text{good }Q\subsetneq Q_0}{\diam Q >\delta}} \ell_Q \cup \bigcup_{\stackrel{\text{good }Q\subseteq Q_0}{\diam Q\leq \delta}}\overline{Q},$$ whence $$\Haus^1_\delta(\overline{T}) \leq \sum_{\stackrel{\text{good }Q\subsetneq Q_0}{\diam Q >\delta}} \diam \ell_Q + \sum_{\stackrel{\text{good }Q\subseteq Q_0}{\diam Q \leq \delta}}\diam \overline{Q}= \sum_{\stackrel{\text{good }Q\subsetneq Q_0}{\diam Q >\delta}} \frac{1}{2}\diam Q+ \sum_{\stackrel{\text{good }Q\subseteq Q_0}{\diam Q \leq\delta}} \diam Q.$$ Here we used the fact that any straight line segment $\ell$ can be subdivided into finitely many line segments $\ell'_1,\dots, \ell'_k$ such that $\diam \ell'_i\leq \delta$ for all $i$ and $\sum_{i=1}^k \diam \ell_i'=\diam \ell$.  Since $\sum_{\text{good }Q\subseteq Q_0} \diam Q < N/\varepsilon$, it follows that

\begin{equation*}\label{e:To} \Haus^1(\overline{T})=\lim_{\delta\downarrow 0} \Haus^1_\delta(\overline{T}) \leq \frac{1}{2}\sum_{\text{good }Q\subsetneq Q_0} \diam Q < \,\frac{N}{2\varepsilon}.\end{equation*}
Now,

\begin{align}
 \notag B&\subseteq Q_0 \setminus\bigcup_{\text{bad }Q\subset Q_0}Q \\
 \label{e:zero} &=\bigcup\left\{\bigcap_{i=0}^\infty Q_i : Q_0\supseteq Q_1\supseteq \cdots\text{ is a chain of good cubes, } \lim_{i\rightarrow\infty} \diam Q_i=0\right\}\\
 \label{e:in-T} &\subseteq
\left\{\lim_{i\rightarrow\infty}x_i:x_i\in \ell_{Q_i}\text{ for some good cubes }Q_0\supseteq Q_1\supseteq \cdots,\  \lim_{i\rightarrow\infty} \diam Q_i= 0\right\}.\end{align}
Thus, $B\subseteq \overline{T}$ by (\ref{e:in-T}). Moreover, refining \eqref{e:zero},
we obtain
$B \subseteq \bigcap_{j=1}^\infty G_j$, where
$$
G_j=\bigcup\left\{\bigcap_{i=j}^\infty Q'_i: Q'_j\supsetneq Q'_{j+1}\supsetneq\cdots\text{ is a chain of good cubes, } \diam Q'_j\leq 2^{-j} \right\}.$$
Since $\sum_{\text{good }Q\subseteq Q_0}\diam Q<\infty$, we have
$\Haus^1_{2^{-j}} (G_j)\to 0$, which implies $\Haus^1(B)=0$.
Finally, because $\overline{T}$ is a continuum in $\RR^n$ with $\Haus^1(\overline{T})<\infty$, $\overline{T}$ coincides with the image $\Gamma=f([0,1])$ of some Lipschitz map $f:[0,1]\rightarrow \RR^n$;
e.g.~see \cite[Theorem I.1.8]{DS93} or \cite[Lemma 3.7]{Schul-Hilbert}.
 \end{proof}

The proof of Theorem \ref{t:fast} uses Lemmas \ref{l:good-bad} and \ref{l:tree} repeatedly over a suitable, countable choice of parameters.

\begin{proof}[Proof of Theorem \ref{t:fast}] Suppose $n\geq 2$ and let $\mu$ be a locally finite Borel measure on $\RR^n$. Our goal is to show that $\mu \res\left\{x\in\RR^n: S(\mu,x)<\infty\right\}$ is $1$-rectifiable. It suffices to prove that $\mu\res\{x\in Q_0: S(\mu,x)\leq N\}$ is $1$-rectifiable for all $Q_0\in\Delta$ and for all integers $N\geq 1$.

Fix $Q_0\in\Delta$ and $N\geq 1$. Let $A=\{x\in Q_0: S(\mu,x)\leq N\}$. If $\mu(A)=0$, then there is nothing to prove. Thus, assume $\mu(A)>0$. Then $\eta=\mu(Q_0)>0$, as well. Pick any sequence $(\varepsilon_i)_{i=1}^\infty$ such that $0<\varepsilon_i<1/\eta$ for all $i\geq 1$ and $\varepsilon_i\rightarrow 0$ as $i\rightarrow\infty$. By Lemmas \ref{l:good-bad} and \ref{l:tree}, there exist a Borel set $B_i=B(\mu,Q_0,N,\varepsilon_i)\subseteq A$ and a rectifiable curve $\Gamma_i\supseteq B_i$ such that $\Haus^1(B_i)=0$ and $\mu(A\setminus B_i) \leq \varepsilon_i\eta\mu(A)$. Hence $$\mu\left(A\setminus\bigcup_{i=1}^\infty\Gamma_i\right) \leq \mu\left(A\setminus \bigcup_{i=1}^\infty B_i\right)
\leq \inf_{j\geq 1} \mu(A\setminus B_j) \leq \eta\mu(A)\,\inf_{j\geq 1} \varepsilon_j=0.$$ Therefore, $\mu\res A$ is $1$-rectifiable, and moreover, $\mu\res A\left(\RR^n\setminus \bigcup_{i=1}^\infty B_i\right)=0$. \end{proof}

\bibliography{sufficient}{}
\bibliographystyle{amsalpha}

\end{document}